\documentclass[12pt]{amsart}
\usepackage{latexsym}
\usepackage{a4}
\usepackage{amssymb}
\usepackage{bm}
\usepackage{algpseudocode} 
\usepackage{enumitem}
\newtheorem{thm}{Theorem}[section] 
 
\newtheorem{lem}[thm]{Lemma} 
 
\newtheorem{cor}[thm]{Corollary} 
\theoremstyle{definition}

\newtheorem{de}[thm]{Definition}
\numberwithin{equation}{section} 

\newcommand{\ab}[1]{{\mathbf{#1}}}

\newcommand{\N}{\Bbb{ N}} 
\newcommand{\Z}{\Bbb{ Z}} 

\newcommand{\setsuchthat}{\mid} 

\newcommand{\congmod}[3]{#1 \equiv #2 \,\, \left(\mbox{\rm mod } #3\right)}

\newcommand{\UComm}[2]{\left[ {#1}, {#2} \right]}

\newcommand{\Pol}{\mathrm{Pol}}
\newcommand{\Clo}{\mathrm{Clo}}
\newcommand{\Clop}{\mathrm{Clop}}

\newcommand{\Hoc}{\mathrm{Hoc}}

\newcommand{\Con}{\mathrm{Con}}

\newcommand{\lcover}{\prec}

\newcommand{\conginfix}[1]{\equiv_{#1}}

\newcommand{\algop}[2]{( {#1}, {#2} )}

\newcommand{\R}{\mathbb{R}}

\renewcommand{\emptyset}{\varnothing}

\newcommand{\KX}{K[x_i \mid i \in \N]}
\newcommand{\Kxn}[1]{K[x_1, \ldots, x_{#1}]}

\title[Nilpotent algebras of prime power order]{Bounding the free spectrum of
  nilpotent algebras of prime power order}

\author{Erhard Aichinger}

\address{Erhard Aichinger,
Institut f\"ur Algebra,
Johannes Kepler Universit\"at Linz,
4040 Linz,
Austria}
\email{\tt erhard@algebra.uni-linz.ac.at}

\subjclass[2010]{08A40 (08B20,20N05)}

\urladdr{http://www.jku.at/algebra}
\thanks{Supported by the Austrian Science Fund (FWF):P29931.}
\keywords{nilpotent algebra, free spectrum, supernilpotent algebra, congruence modular variety}
\date{\today}

\begin{document}
\bibliographystyle{amsalpha}
\begin{abstract}
  Let $\ab{A}$ be a finite nilpotent algebra in a congruence modular variety
  with finitely many fundamental operations. If $\ab{A}$ is of
  prime power order, then it is known that
  there is a polynomial $p$ such that for every $n \in \N$,
  every $n$-generated algebra in the variety generated by $\ab{A}$ has
  at most $2^{p(n)}$ elements.
  We present a bound on the degree of this polynomial.  
\end{abstract}

\maketitle
\section{Introduction} \label{sec:intro}

The binary commutator operation defined by \cite{Sm:MV} and studied
in \cite{FM:CTFC, MMT:ALVV} has allowed to generalize concepts from
group theory, such as solvability or nilpotency, from groups to
arbitrary universal algebras. For an algebra $\ab{A}$ 
in a congruence modular variety, its \emph{lower central series} is a series
of its congruence relations, and it is defined 
by $\lambda_1 := 1_A$ and $\lambda_{k+1} := [1_A, \lambda_k]$ for $k \in \N$, where
$[. \, , \, .]$ denotes the \emph{term condition commutator} defined in \cite{FM:CTFC, MMT:ALVV}.
If $\lambda_{k+1} = 0_A$, then $\ab{A}$ is called \emph{$k$-nilpotent}.
From \cite{Hi:TOOR}, we know that for a $k$-nilpotent group
$\ab{G}$, there is a polynomial $p$ of degree $k$ such that
for all $n \in \N$, 
all $n$-generated groups in the variety generated by $\ab{G}$ are
of size at most $2^{p(n)}$. This property can be investigated
for arbitrary algebraic structures, and we say that a finite
algebra $\ab{A}$ has \emph{small free spectrum} if
there is a polynomial $p$ such that for all $n \in \N$,
every $n$-generated algebra in the variety generated by $\ab{A}$ is
of size at most $2^{p(n)}$.
Straightforward generalizations of the group theoretic
results fail:
In \cite[p.\ 308, Example 2]{VL:NIPV} Vaughan-Lee constructed a nilpotent loop of size  $12$,
and \cite[p.\ 283]{AM:PCOG} exhibits a nilpotent expansion of  the six element abelian group with one
unary operation, which both fail to have small free spectrum.
However, in a congruence modular variety, the following result is known:
\begin{thm}[{\cite[Theorem~2]{BB:FSON}}]   \label{thm:fund}
  Let $\ab{A}$ be a finite nilpotent
  algebra of finite type in a congruence modular variety.
  We assume that $\ab{A}$ is a direct product of algebras of prime power order.
  Then $\ab{A}$ has small free spectrum.
\end{thm}
If $\ab{A}$ is a group, this is known from \cite{Hi:TOOR}. The proof of the above theorem
relies on a generalization of Higman's combinatorial argument
given in  \cite{BB:FSON}  and on bounding the rank of 
the commutator terms of $\ab{A}$. Such a bound was derived
in \cite{VL:NIPV} and
Chapter~14 of  \cite{FM:CTFC} in the course of proving that
an  algebra satisfying the assumptions of Theorem~\ref{thm:fund}
has a finite
basis for its equational laws.
In other words, the above theorem by Berman and Blok tells that
for each algebra $\ab{A}$ satisfying the assumptions of Theorem~\ref{thm:fund},
there is a polynomial $p$ such that every $n$-generated algebra
in the variety generated by $\ab{A}$ has at most $p(n)$ elements.
The contribution of the present work is an upper bound on the
degree of $p$. In deriving this upper bound, we obtain an
alternative proof of Theorem~\ref{thm:fund}.
We observe that for a finite algebra $\ab{A}$, every $n$-generated
algebra in the variety generated by $\ab{A}$ is a homomorphic
image of the free algebra in this variety, and this free algebra
is isomorphic to the algebra $\Clo_n (\ab{A})$ of $n$-ary term
functions on $\ab{A}$, and the  \emph{free spectrum} $f_{\ab{A}}$
of $\ab{A}$ is defined by $f_{\ab{A}} (n) := |\Clo_n (\ab{A})|$.
We also mention that Theorem~3.14 from \cite{Ke:CMVW} provides
some kind of a converse: a finite algebra in a congruence modular
variety with small free spectrum is a direct product of algebras of prime
power order.

The property of having  small free spectrum  is closely
related to \emph{supernilpotency}, a notion introduced in \cite{AE:EKCN,
AM:SAOH}.
We say that an algebra $\ab{A}$  is \emph{$k$-supernilpotent}
if the higher commutator operation defined in \cite{Bu:OTNO}
and studied, e.g., in \cite{AM:SAOH,Mo:HCT}
satisfies $[1_A, \ldots, 1_A]_{\ab{A}} = 0_A$ ($k+1$ repetitions of $1_A$);
this
condition is formulated without using higher commutators in
Definition~\ref{de:tc} below. The algebra
$\ab{A}$ is called \emph{supernilpotent} if there is $k \in \N$ such that
$\ab{A}$ is $k$-supernilpotent.
For those classes of algebra that we will study here, supernilpotency
implies nilpotency: this implication holds in congruence
permutable varieties by \cite{AM:SAOH}, and more generally
in congruence modular varieties by \cite{Wi:OSA}.
The connection between \emph{supernilpotency} and \emph{small free spectrum}
is stated in Lemma~\ref{lem:snp} below. From this Lemma,
we see that
a finite algebra $\ab{A}$ in a congruence modular variety
is $k$-supernilpotent if and only
if there is a polynomial $p$ of degree $k$ such that
for its free spectrum, we have  $f_{\ab{A}} (n) \le 2^{p(n)}$ for
all $n \in \N$;
hence $\ab{A}$ is supernilpotent if and only if
$\ab{A}$ has small free spectrum.
Using the concept of supernilpotency,
the theorem by Berman and Blok can be rephrased
as ``every nilpotent algebra
of finite type and prime power order in a congruence modular
variety is supernilpotent''. However, although \cite{BB:FSON}
yields the existence of a $k$ such that the algebra
is $k$-supernilpotent, no explicit
upper bound for $k$  has been computed.
For groups and rings, $k$ can be chosen to be the
nilpotency degree, but this does not hold in
general: for every  $k, m \in \N$ with $m \ge 2$,
\cite{AM:OVCO} exhibits
a $k$-nilpotent algebra of size $2^k$ with fundamental operations
of arity at most $m$ that is $m^{k-1}$-supernilpotent, but
not $(m^{k-1}-1)$-supernilpotent.
These examples  show that a bound on the supernilpotency degree cannot
be a function of $k$ alone, but must contain more information
on the algebra.
For certain algebras (groups expanded with multilinear
operations), an explicit bound was given in  \cite{AM:OVCO}.
Our main theorem provides such a bound for all
algebras covered by the Berman-Blok-Theorem; in particular,
it applies to nilpotent loops of prime power order.
One ingredient used in this bound is the \emph{height} of the
congruence lattice of $\ab{A}$, which we define as
the maximal size of a linearly ordered subset of the lattice
minus one; hence the 
height of the $1$-element lattice is $0$ and the height
of a linearly ordered set with $n$ elements is $n-1$.
\begin{thm} \label{thm:fund2}
  Let $q>1$ be a prime power, let $m \in \N$, and let $\ab{A}$ be
  a nilpotent algebra in a congruence modular variety
  with $|A| = q$ such that all fundamental
  operations of $\ab{A}$ are of arity at most $m$.
  Let $h$ be the height of the congruence lattice of $\ab{A}$, and
  let \[
         s := \big( m (q - 1) \big)^{h-1}.
         \]
         Then $\ab{A}$ is $s$-supernilpotent, and there is
         a polynomial $p \in \R[x]$ of degree at most $s$ such that
         the free spectrum satisfies $f_{\ab{A}} (n) = 2^{p(n)}$ for
         all $n \in \N$. 
\end{thm}
From this result, we obtain the following improvement of Theorem~\ref{thm:fund}.
\begin{cor} \label{cor:cmv} 
  Let $\ab{A}$ be a finite nilpotent algebra 
  in a congruence modular
  variety that is a direct product of algebras of prime power order,
  and let $m \in \N$ be such that
   such that all fundamental
operations of $\ab{A}$ are of arity at most $m$. We assume $|A|> 1$.
Let
\[
s := \big(m (|A| - 1)\big)^{(\log_2 (|A|)-1)}.
\]
Then
$\ab{A}$ is $s$-supernilpotent
and 
there is a polynomial $p \in \R[x]$ of degree $\le s$ such that the free spectrum
  satisfies
  $f_{\ab{A}} (n)= 2^{p(n)}$ for all $n \in \N$.
\end{cor}
Combining this with \cite{Ke:CMVW}, we obtain:
\begin{cor} \label{cor:dec}
   Let $\ab{A}$ be a finite algebra 
   in a congruence modular
   variety with $|A| > 1$,
   and let $m \in \N$ be such that
   such that all fundamental
   operations of $\ab{A}$ are of arity at most $m$. Then we have:
   \begin{enumerate}
     \item \label{it:dec1}
     If $\ab{A}$ has small free spectrum, then there
     is a polynomial $p \in \R[x]$ of degree at most
     $(m (|A|-1))^{(\log_2 (|A|)-1)}$ such that $f_{\ab{A}} (n) = 2^{p(n)}$ for all
     $n \in \N$.
   \item \label{it:dec2}
     If $\ab{A}$ is supernilpotent, then it is
     $\big((m (|A|-1))^{(\log_2 (|A|)-1)}\big)$-supernilpotent.
   \end{enumerate}
\end{cor}   
The proofs of these results will be given in Section~\ref{sec:proofs}.     
Our proof of Theorem~\ref{thm:fund2} will proceed as follows:
We define a binary operation $+$ on  $\ab{A} = \algop{A}{F}$ such that
$\algop{A}{+}$ is an elementary abelian group and
$\ab{A'} = \algop{A}{F \cup \{+\}}$ is still nilpotent.
Since $\algop{A}{+}$ is elementary abelian, we can expand
it to a finite field $\algop{A}{+, \cdot}$ and represent
all fundamental operations from $\ab{A}$ by polynomials over
this field. Using this representation, we show that
$\ab{A'}$ is $s$-supernilpotent, which implies
that its reduct $\ab{A}$ is also $s$-supernilpotent.

\section{Preliminaries about supernilpotency} \label{sec:prelsnp}

We use the definition of supernilpotency  in
\cite[Definition~7.1]{AM:SAOH}.
This definition can be stated as follows:
\begin{de}[Term condition for supernilpotency] \label{de:tc}
  Let $\ab{A}$ be an algebra and $k \in \N$. Then $\ab{A}$ is
  \emph{$k$-supernilpotent} if for all $n_1, \ldots, n_{k+1} \in \N_0$ and
  for all $\langle (a^{(i)}_1, a^{(i)}_2) \mid i \in \{1,\ldots, k + 1\} \rangle
  \in \prod_{i=1}^{k+1} (A^{n_i} \times A^{n_i})$ and for all
  $\sum_{i=1}^{k+1} n_i$-ary term functions $t$ of $\ab{A}$ the following holds:
  if for all $f : \{1,\ldots, k\} \to \{1,2\}$ such that
  $f$ is not constantly $2$, we have
  \[
  t( a_{f(1)}^{(1)}, \ldots, a_{f(k)}^{(k)}, a_{1}^{(k+1)}) =
  t( a_{f(1)}^{(1)}, \ldots, a_{f(k)}^{(k)}, a_{2}^{(k+1)}),
  \]
  then
  \[
  t (a_{2}^{(1)}, \ldots, a_{2}^{(k)},  a_{1}^{(k+1)})
  =
      t (a_{2}^{(1)}, \ldots, a_{2}^{(k)},  a_{2}^{(k+1)}).
      \]
\end{de}
From this definition, we see immediately that reducts of
supernilpotent algebras are supernilpotent:
\begin{lem} \label{lem:reduct}
  Let $s \in \N$, and
  let $\ab{A}, \ab{B}$ be universal algebras with the same
  universe. If $\ab{B}$ is $s$-supernilpotent
  and the clones of term operations of these
  algebras satisfy $\Clo (\ab{A}) \subseteq \Clo (\ab{B})$,
  then $\ab{A}$ is also $s$-supernilpotent.
\end{lem}
We also see that $s$-supernilpotency is defined by an infinite
set of quasi-identities, and is therefore preserved under
taking subalgebras and direct products.

If $\ab{A} = \algop{A}{+,-,0, (f_i)_{i \in I}}$ is an expanded group,
we can describe supernilpotency more easily. For $n \in \N$,
we call
a function $f : A^n \to A$ \emph{absorbing}
if for all $a_1, \ldots, a_n \in A$ with
$0 \in \{a_1, \ldots, a_n \}$, we have
$f (a_1, \ldots, a_n) = 0$. The prototypes of absorbing functions
are the commutator $(a_1, a_2) \mapsto -a_1 - a_2 + a_1 + a_2$ in
any group,
$(a_1, a_2) \mapsto a_1a_2$ in any ring, and, also on every ring,
every function that can be written as $(a_1, \ldots, a_n) \mapsto
a_1a_2\cdots a_n \cdot g(a_1, \ldots, a_n)$.
The \emph{essential arity} of $f : A^n \to A$ is the number of arguments
on which $f$ depends.
We note that the essential arity of an absorbing function $f : A^n \to A$
is either $n$ or $0$.

\begin{lem} \label{lem:snpeg}
  Let $\ab{A} = \algop{A}{+,-,0, (f_i)_{i \in I}}$ be an expanded group, and
  let $s \in \N$.
  Then the following are equivalent:
  \begin{enumerate}
  \item \label{it:s1}  $\ab{A}$ is $s$-supernilpotent.
  \item \label{it:s2} All absorbing polynomial functions of $\ab{A}$ are of
     essential arity at most $s$.
  \end{enumerate}
  If $\ab{A}$ is finite, then \eqref{it:s1} and \eqref{it:s2} are furthermore equivalent to
  \begin{enumerate}
    \setcounter{enumi}{2}
  \item \label{it:s3} There is a polynomial $p \in \R[x]$ of degree $\le s$ such
    that $f_{\ab{A}} (n) = |\Clo_n (\ab{A})| =
    2^{p(n)}$ for all $n \in \N$.
  \end{enumerate}   
\end{lem}
\emph{Proof:}
The equivalence of~\eqref{it:s1} and~\eqref{it:s2} follows from
Corollary~6.12 of \cite{AM:SAOH} by observing that $s$-supernilpotency
is equivalent to the higher commutator property
$[1_A, \ldots, 1_A] = 0_A$ ($(s+1)$ times $1_A$).
The equivalence of~\eqref{it:s3} and~\eqref{it:s1}
follows from Corollary~4.3 of \cite{Ai:OTDD};
there it 
was  proved using a modification of an argument that goes
back to \cite{Hi:TOOR}. \qed

The equivalence of~\eqref{it:s1} and~\eqref{it:s3} is actually
true
for all finite algebras in congruence modular varieties.
Following \cite{FM:CTFC},
we say that a term $w (x_1, \ldots, x_{r+1})$
in the language of $\ab{A}$ is a \emph{commutator term}
of rank $r$ for $\ab{A}$ if 
 $\ab{A} \models w(z,x_2, \ldots, x_{r}, z) \approx
      w(x_1,z, \ldots, x_{r}, z) \approx \dots \approx
      w(x_1, x_2, \ldots, z, z) \approx z$.
 A commutator term $w (x_1, \ldots, x_{r+1})$ is called \emph{trivial} 
      if $\ab{A} \models w(x_1, \ldots, x_{r}, z) \approx z$.
A part of the next lemma has also been stated in~\cite{AMO:COTR}. 
\begin{lem} \label{lem:snp}
  Let $\ab{A}$ be a finite algebra in a congruence modular variety,
  and 
  let $s \in \N$.
  Then the following are equivalent:
  \begin{enumerate}
  \item \label{it:s1}  $\ab{A}$ is $s$-supernilpotent.
  \item \label{it:s2}  $\ab{A}$ is nilpotent, and
    all nontrivial commutator terms of $\ab{A}$
    are of rank at most $s$.
  \item \label{it:s3} There is a polynomial $p \in \R[x]$ of degree at most $s$ such
                       that $f_{\ab{A}} (n) = 2^{p(n)}$ for all $n \in \N$.    
  \item \label{it:s4}  $\lim_{n \to \infty} \big( \log_2 (f_{\ab{A}} (n)) / n^{s+1} \big) = 0$.
\end{enumerate}   
\end{lem}
\emph{Proof:}
\eqref{it:s1}$\Rightarrow$\eqref{it:s2}:
By
\cite[Theorem~4.11]{Wi:OSA}, $\ab{A}$ has an $(s+1)$-difference term~$d$.
Since $\ab{A}$ is
$s$-supernilpotent, $d$ is a Mal'cev term.
From
\cite[Lemma~7.5]{AM:SAOH}, we obtain that $\ab{A}$ is nilpotent
and all commutator terms have rank at most $s$.
This bound on the rank can also be seen directly from the term
condition that defines supernilpotency: to this end,
let $w(x_1,\ldots, x_{r+1})$ be a commutator term of $\ab{A}$ with $r > s$.
We want to show that $w$ satisfies $\ab{A} \models w(x_1, \ldots, x_r, z) \approx z$.
To this end, let $\xi_1, \ldots, \xi_r, \zeta \in A$.
We apply the term condition from Definition~\ref{de:tc} with the
following settings:
$t := w^{\ab{A}}$, 
$a_1^{(i)} := \zeta$ and $a_2^{(i)} := \xi_i$ for $i \in \{1,\ldots, s \}$,
$a_1^{(s + 1)} := (\zeta, \ldots, \zeta)$ ($r-s+1$ times $\zeta$), $a_2^{(s)} := (\xi_{s+1}, \ldots, \xi_r, \zeta)$.
Then the term condition implies
$t(\xi_1, \ldots, \xi_s, \zeta, \ldots, \zeta) =
t(\xi_1, \ldots, \xi_s, \ldots, \xi_r, \zeta)$.
Since $w$ is a commutator term, $\zeta = t(\xi_1, \ldots, \xi_s, \zeta, \ldots, \zeta)$.
Thus $\ab{A} \models w(x_1, \ldots, x_r, z) \approx z$, and hence $w$ is trivial.

\eqref{it:s2}$\Rightarrow$\eqref{it:s3}:
Under the additional assumption that $\ab{A}$ is a direct product of algebras
of prime power order, this is shown in the proof of Theorem~2 of \cite{BB:FSON}.
However, this additional assumption is only used to obtain a bound on the rank
of nontrivial commutator terms, which is claimed by~\eqref{it:s2}.

      \eqref{it:s3}$\Rightarrow$\eqref{it:s4}: Obvious.

      \eqref{it:s4}$\Rightarrow$\eqref{it:s1}:
      The proof for this implication comes from
      \cite{AMO:COTR}; it is included for easier reference.
      
  Theorem~9.18 of \cite{HM:TSOF} implies
 that the variety $V (\ab{A})$ omits types $\boldsymbol{1}$ and $\boldsymbol{5}$.
 From \cite[Lemma~12.4]{HM:TSOF}, we obtain that $\ab{A}$ is right nilpotent,
 and since the commutator operation in a congruence modular variety is commutative,
 $\ab{A}$ is therefore nilpotent. Now \cite[Theorem~6.2]{FM:CTFC} yields that 
 $\ab{A}$ has a Mal'cev term. 
 Let $\ab{A}^*$ be the expansion of $\ab{A}$ with all its constants.
 Then $\ab{A}^*$ is nilpotent and generates a congruence permutable variety.
  The variety $V (\ab{A}^*)$ is nilpotent by \cite[Theorem~14.2]{FM:CTFC}, and hence 
 congruence uniform by \cite[Corollary~7.5]{FM:CTFC}.
 Since for all $n \in \N$,   $f_{\ab{A}^*} (n) \le f_{\ab{A}} (n + |A|)  \le
 2^{p (n + |A|)}$, we obtain from the proof of \cite[Theorem~1]{BB:FSON} that
 all commutator terms of $\ab{A}^*$ are of rank
 at most $s$. Hence all commutator polynomials (in the sense of \cite[Definition~7.2]{AM:SAOH})
 of $\ab{A}$
 are of rank at most $s$, and then \cite[Lemma~7.5]{AM:SAOH} yields that $\ab{A}$ is $s$-supernilpotent.
 \qed
 
 It is worth noting that in proving~Lemma~\ref{lem:snp}, we needed to employ
 substantial results from each of the sources
 \cite{BB:FSON, FM:CTFC, HM:TSOF, AM:SAOH, Wi:OSA}.

 \section{Preliminaries on commutators and nilpotency} 
 In this section, we compile some well known facts on
 the relation between the commutator operation and
 the Mal'cev term of an algebra.
 This is an extension of \cite[p.\ 14]{Ai:TPFO2}.
 Let $\ab{A}$ be an algebra with a Mal'cev term $d$.
We fix an element $o \in A$ and define two binary operations $+_o$ and $-_o$ by
\begin{equation} \label{eq:deplus}
    \begin{array}{rcl}
        x +_o y & := & d (x,o,y) \text{ and } \\
        x -_o y & := & d (x,y,o) \text{ for } x,y \in A.
    \end{array}
\end{equation}
Sometimes, we also use $-_o$ as a unary operation: then
  $-_o \, y$ stands for $o -_o y = d(o, y, o)$.
In the following proposition, we compile those relations of   
$+_o$ and $-_o$ with the commutator that we will need in the sequel.
Such properties have been established from the very beginning
of modular commutator theory (cf. \cite{He:AAIC, Gu:GMIC}), and
the proofs of several of these properties are taken from \cite{Ai:TPFO2}.
The proofs given below rely only on the following fact that follows rather directly
from the definition of the term condition defining the binary commutator operation
(see Lemma~2.2 of \cite{Ai:TPFO2} or
Exercise~4.156(2) from \cite{MMT:ALVV}): if  $\alpha$ and $\beta$ are congruences of
any algebra $\ab{A}$,
$(a,b) \in \alpha$, $(c,d) \in \beta$, $p \in \Pol_2 (\ab{A})$, and
$p (a,c) = p (a, d)$, then $\congmod{p(b,c)}{p(b,d)}{\UComm{\alpha}{\beta}}$.
\begin{lem}[cf. {\cite[Proposition~2.7]{Ai:TPFO2}}]    \label{lem:plus1}
   Let $\ab{A}$ be an algebra with a Mal'cev term
   $d$,
   let $a,b,c,o$ be elements of $A$, let $+_o$ and
   $-_o$ be defined as in~\eqref{eq:deplus}, 
   and let $\alpha, \beta$ be congruences of 
   $\ab{A}$.
   Then we have:
   \begin{enumerate}
      \item \label{it:p1} $a +_o o = o +_o a = a -_o o = a$.
      \item \label{it:p2} $a -_o a = o$.
      \item \label{it:p3} If $a \conginfix{\alpha} b \conginfix{\beta} o$,
                          then 
      \(
          \congmod{(a -_o b) +_o b}{a}{\UComm{\alpha}{\beta}}.
      \)        
      \item \label{it:p4} If $a \conginfix{\alpha} o  \conginfix{\beta} b$,
             then
        \(
          \congmod{(a +_o b) -_o b}{a}{\UComm{\alpha}{\beta}}.
        \)
      \item \label{it:p6} If $a \conginfix{\alpha} o  \conginfix{\beta} b$,
        then $\congmod{a +_o b}{b +_o a}{[\alpha, \beta]}$.
      \item \label{it:p7} If $a \conginfix{\alpha} o \conginfix{\beta} b$, then
        $\congmod{ (a +_o b) +_o c}{a +_o (b +_o c)}{\UComm{\alpha}{\beta}}$.
      \item \label{it:p8} If $a \conginfix{\alpha} o \conginfix{\beta} b$, then
        $\congmod{d (a +_o b, b, c)}{a +_o c}{\UComm{\alpha}{\beta}}$.
      \item \label{it:p9} If $a \conginfix{\alpha} o$, then
        $\congmod{(-_o \, a) +_o a}{o}{\UComm{\alpha}{\alpha}}$.
   \end{enumerate}
\end{lem}
\emph{Proof:} Properties~\eqref{it:p1} and~\eqref{it:p2} follow from
the properties of the Mal'cev term $d$.
For proving~\eqref{it:p3}, we define a polynomial function $t \in \Pol_2 (\ab{A})$ by
\(
    t(x,y) := d ( d (a,x,y), y, b).
\)
We have $t(a,o) = t(a,b) = b$. Thus
we obtain $\congmod{t(b,o)}{t(b,b)}{\UComm{\alpha}{\beta}}$, which means
\(
   \congmod{(a -_o b) +_o b}{a}{\UComm{\alpha}{\beta}}.
\)
For proving~\eqref{it:p4}, we define a polynomial function $t \in \Pol_2 (\ab{A})$ by
\(
   t(x,y) :=  d ( d ( x, y, b ), d (b, y, o) , o).
\) 
We have $t(o,o) = t(o, b)  = o$. 
Thus
we obtain $\congmod{t(a,o)}{t(a,b)}{\UComm{\alpha}{\beta}}$, which means
\(
   \congmod{(a +_o b) -_o b}{a}{\UComm{\alpha}{\beta}}.
\)
For proving~\eqref{it:p6}, we define
$t (x, y) := d (y +_o x, x +_o y, a +_o b)$ for $x,y \in A$.
Then we have $t (o, o) = t (o, b) = a +_o b$, and therefore
$\congmod{t (a,o)}{t(a,b)}{\UComm{\alpha}{\beta}}$, which
implies $\congmod {a +_o b}{b +_o a}{\UComm{\alpha}{\beta}}$. 
For proving~\eqref{it:p7}, we define
$t(x,y) := d \big( x +_o (y +_o c),  (x +_o y) +_o c,  (a +_o b) +_o c) \big)$
and have $t(o,o) = t (o, b) = (a +_o b) +_o c$, and therefore
$\congmod{t (a,o)}{t(a,b)}{\UComm{\alpha}{\beta}}$, which
implies $\congmod {(a +_o b) +_o c}{a +_o (b +_o c)}{\UComm{\alpha}{\beta}}$.
For proving~\eqref{it:p8}, we consider the polynomial function of $\ab{A}$ defined
                 by
                 \(
                 t (x,  y) := d (x  +_o  y,  y, c)
                 \)
                 for $x , y  \in A$. Then $t (o,o) = d (o, o, c) = c$ and
                 $t (o, b) = d(b, b, c) = c$. Therefore,
                 $(t (a, o), t (a, b)) \in \UComm{\alpha}{\beta}$, and thus
                 $\congmod{d (a, o ,c)}{d(a +_o b, b, c)}{\UComm{\alpha}{\beta}}$.
                 Since $d (a, o, c) = a +_o c$, the result follows.
                 For property~\eqref{it:p9},
                 we observe that
                 $(-_o \, a) +_o a =
                 (o -_o a) +_o a$. By property~\eqref{it:p3},
                 the last expression is congruent modulo $[\alpha,\alpha]$
                 to $o$.
           \qed
           
      The following well known Lemma goes back to \cite{He:AAIC, Fr:SIAI}.
      \begin{lem}  \label{lem:abgroup}
        Let $\ab{A}$ be an algebra with Mal'cev term $d$, and let
             $\alpha$ be a congruence of $\ab{A}$ with $[\alpha, \alpha] = 0_A$.
             Let $Q := o / \alpha$. Then $\ab{Q} := \algop{Q}{+_o, -_o, o}$ is an abelian group.
             If $\alpha$ is furthermore a minimal congruence of $\ab{A}$ and $Q$ is finite, then
             $\ab{Q}$ is of prime exponent.
           \end{lem}
           \emph{Proof:}
           The first part follows from items \eqref{it:p7},\eqref{it:p1}, and \eqref{it:p9} of
           Lemma~\ref{lem:plus1}. For the second part, we sketch an argument taken from
           \cite[p.\ 151]{Fr:SIAI}:
           From Proposition~2.8(2) of \cite{Ai:TPFO2}, it is not hard to infer that the group $\ab{Q}$ can be
           seen as a
           module over the finite ring
           $\algop{ \{ p|_Q : p \in \Pol_1 (\ab{A}), p (o) = o\}}{+_o, \circ}$. Since $\alpha$ is a minimal congruence,
           this module has no submodules, and thus $\ab{Q}$ is the additive group of a finite simple module, and has
           therefore prime exponent. \qed

           We will also use the following relational description of centrality that goes back to
           \cite{Ki:TROT}. We call a congruence relation $\zeta$ of $\ab{A}$ \emph{central} in
           $\ab{A}$ if $[\zeta, 1_A] = 0_A$ (cf. \cite[Definition~13.1]{BS:ACIU}).
           \begin{lem}[Relational description of centrality,
               {cf.\  Theorem~3.2(iii) of \cite{Ki:TROT}}]
                \label{lem:kiss}
             Let $\ab{A}$ be an algebra with a Mal'cev term $d$, and let $\zeta \in \Con (\ab{A})$.
             Then $\zeta$ is central in $\ab{A}$ if and only if all fundamental operations of $\ab{A}$
             preserve the relation $\rho = \{ (a_1,a_2,a_3,a_4) \in A^4 \mid (a_1,a_2) \in \zeta, \,
             d (a_1, a_2, a_3) = a_4 \}$.
            \end{lem}
           \emph{Proof:} The result is a special case of \cite[Proposition~2.3 and Lemma~2.4]{AM:PCOG}. \qed

           In expanded groups, the commutator of two congruences can be calculated
           from the associated $0$-classes (ideals) and binary polynomial functions
           \cite[Lemma~2.9]{AM:PCOG}.
           We will only use the following assertion:
           \begin{lem} \label{lem:cideals}
             Let $\ab{A} = \algop{A}{+,-,0, (f_i)_{i \in \N}}$ be an expanded group,
            let $\xi, \eta$ be congruences of $\ab{A}$,
            let $X := 0/ \xi$ and $Y := 0/\eta$ be the ideals of $\ab{A}$ associated
            with these congruences, and let $p \in \Pol_2 (\ab{A})$ such that
            $p(a,0) = p(0,a) = 0$ for all $a \in A$.
            Then for all $x \in X$ and $y \in Y$, $\congmod{p(x,y)}{0}{[\xi, \eta]}$,
            and therefore $p(x,y)$ lies in the ideal $[X,Y] := 0/[\xi, \eta]$ associated with $[\xi, \eta]$.
          \end{lem}   
          \emph{Proof:}  
          Since $p(0,0) = p (0,y)$, the term condition yields
          $\congmod{p(x,0)}{p(x,y)}{[\xi, \eta]}$. \qed
          
\section{Expanding an algebra with a group operation}
Let $\ab{A}$ be an algebra in a congruence modular
variety, let $m \in \N_0$,  and let $0_A = \alpha_0 \le  \alpha_1 \le \dots \le \alpha_m = 1_A$
be a linearly ordered sequence of equivalence relations on $A$.
$L = \langle \alpha_i \mid i \in \{0,1,\ldots,m\} \rangle$ is a \emph{central series} of $\ab{A}$
if for each $i \in \{1,\ldots, m\}$, $\alpha_i$ is a congruence relation of $\ab{A}$ and
$\alpha_i/\alpha_{i-1}$ is
central in $\ab{A}/\alpha_{i-1}$; using the homomorphism property of
the modular commutator,
this centrality can be expressed by $[1_A, \alpha_i] \le \alpha_{i-1}$.
An algebra is nilpotent if and only if it has a finite central series.
We fix an element $o \in A$ and a Mal'cev term $d$ of $\ab{A}$. 
For each $i \in \{1, \ldots, m\}$, we let $G_i \subseteq \ab{A}/\alpha_{i-1}$
be defined by
\[
G_i := \{ x/ \alpha_{i-1} \mid x \in A, (x, o) \in \alpha_i \}.
\]
In other words, $G_i$ is the image of $o/\alpha_i$ under the canonical
projection from $\ab{A}$ to $\ab{A}/\alpha_{i-1}$. Let $\bar{o} :=
o/\alpha_{i-1}$.
Since $[\alpha_i, \alpha_i] \le [1, \alpha_{i}] \le  \alpha_{i-1}$, Lemma~\ref{lem:abgroup}
and the homomorphism property of the modular commutator tell that the operations
$g +^i h := d(g,\bar{o},h)$ and $-^i \, g := d(\bar{o},g,\bar{o})$ turn $G_i$ into an abelian group.
We call \[ \ab{G} := \prod_{i=1}^m \algop{G_i}{+^i,-^i,\bar{o}} \] \emph{the abelian group associated
  with the algebra $\ab{A}$, its central ceries $L$, and zero $o$.}
\begin{lem} \label{lem:maxchain}
  Let $\ab{A}$ be a finite nilpotent algebra in a congruence modular variety,
  let $0_A = \alpha_0 \lcover \alpha_1 \lcover \cdots \lcover \alpha_m = 1_A$
  be a maximal chain in the congruence lattice of $\ab{A}$, and let $o \in A$.
  Let $\ab{G} = \prod_{i=1}^m \ab{G}_i$ be  the abelian group associated with $\ab{A}$, $L := \langle \alpha_i \mid
  i \in \{0,1,\ldots, m\} \rangle$,   and zero $o$.
  Then $|G|$ is a direct product of groups of prime order, and
  $|\ab{G}| = |\ab{A}|$. If $|\ab{A}|$ is furthermore of prime power order,
  then $\ab{G}$ is elementary abelian.
\end{lem}  
\emph{Proof:}
Since $\ab{A}$ is nilpotent,
      $L$ is a central series of $\ab{A}$.
By Lemma~\ref{lem:abgroup}, each of the groups $\ab{G}_i$ is an abelian group of prime
exponent. Furthermore, $\ab{A}$ is congruence uniform \cite[Corollary~7.5]{FM:CTFC}.
For proving $|\ab{G}| = |\ab{A}|$, we proceed by induction on
$m$ and have $|\ab{A}| = |o/\alpha_1| \cdot |\ab{A}/\alpha_1| =
|\ab{G}_1| \cdot |\ab{A}/\alpha_1|$, which is equal to 
to $|\ab{G}_1| \cdot  \prod_{i=2}^m |\ab{G}_i| = |\ab{G}|$ by the induction hypothesis.
Now assume that $\ab{A}$ is of prime power order. Then
$\ab{G}$ is of prime power order and squarefree exponent, and therefore
elementary abelian.
\qed

The following theorem allows to expand a nilpotent algebra with a Mal'cev term
with group operations such that nilpotency is preserved.
\begin{thm} \label{thm:expand}
  Let $\ab{A} = \algop{A}{F}$ be a
  nilpotent algebra with
  a Mal'cev term $d$, let $m \in \N$, let $o \in A$,
  and
  let $L = \langle \alpha_i \mid i \in \{0,\ldots, m\} \rangle$
  be a central series of $\ab{A}$. 
  Then there exist a binary function $+: A\times A \to A$ and a unary function
  $-: A \to A$ such that 
  \begin{enumerate}
    \item \label{it:t1} $\algop{A}{+,-,o}$ is isomorphic to the abelian
      group $\ab{G}$ associated with $\ab{A}$, $L$, and $o$.
    \item \label{it:t2} $L = \langle \alpha_i \mid i \in \{0,\ldots, m\} \rangle$
      is a central series also of the expansion $\ab{A'} = \algop{A}{F \cup \{+, -,o\}}$ of $\ab{A}$,
      and therefore $\ab{A'}$ is nilpotent of class at most $m$.
  \end{enumerate}     
  \end{thm}     
  
\emph{Proof:} 
As in \eqref{eq:deplus}, we define
$x +_o y  :=  d (x,o,y)$, $x -_o y  :=  d (x,y,o)$, and
$-_o \, y := d (o,y,o)$ for $x,y \in A$. 
  We proceed by induction on $m$.
  We show that there exist $+,-$ such that
  \begin{enumerate}[label={(\roman*)}, ref={\roman*}]
      \item \label{it:f1} for each $i \in \{0,1,\ldots, m\}$, both  $+$ and $-$ preserve the congruence $\alpha_i$,
       \item \label{it:f3} the algebra $\algop{A}{+,-,o}$ is isomorphic to the abelian
    group $\ab{G}$ associated with $\ab{A}$, $L$ and zero $o$,
   \item \label{it:f2} for each $i \in \{1,\ldots, m\}$, both $+$ and $-$ preserve the
    relation $\gamma_i$ given by
    \[
       \gamma_i =  \{ (x_1,x_2,x_3,x_4) \in A^4 \mid
                      (x_1, x_2) \in \alpha_{i},
                      \big(d(x_1,x_2,x_3), x_4\big) \in \alpha_{i-1} \}.
       \]
  \end{enumerate}
  If $m = 0$, then $|A| = 1$. Defining $+$ as the only binary and $-$ as the
  only unary operation of this set, we see that
  $\algop{A}{+,-,o}$ is a one element group, and hence isomorphic
  to the one element group $\ab{G}$.
  
  Now we assume $m \ge 1$.   
  Let $\alpha := \alpha_1$.
    Then $\ab{A}/\alpha$ has a central series
$L_1 = \langle 0_{A/\alpha} = \alpha_1/\alpha, \alpha_2/\alpha, \ldots, \alpha_m/\alpha = 1_{A / \alpha} \rangle$
which is shorter than $L$,
and so we may apply the induction hypothesis on $\ab{A}/\alpha$ 
to obtain 
$\oplus : \ab{A}/\alpha \times \ab{A}/\alpha \to \ab{A} / \alpha$
and $\ominus : \ab{A} / \alpha \to \ab{A} / \alpha$ such that
$\algop{A/ \alpha}{\oplus,\ominus, o/\alpha}$ is isomorphic to
the abelian group associated to $\ab{A}/\alpha$,  $L_1$, and zero $o/\alpha$.
Furthermore, $\oplus$ and $\ominus$ preserve
all congruences in $L_1$  and, for each $i \in \{2, 3, \ldots, m\}$, the
relation
\begin{multline} \label{eq:di}
  \delta_i := \{ (y_1,y_2,y_3,y_4) \in (A/\alpha)^4 \mid
         (y_1, y_2) \in \alpha_i / \alpha, \\
         \big(d(y_1, y_2, y_3), y_4 \big) \in \alpha_{i-1} / \alpha \}.
\end{multline}
Now let $Q := o/\alpha$.
We choose $R$ to be a set of representatives of $A$ modulo $\alpha$ with
$o \in R$, and we let $r : A \to R$ be the function that assigns
to each $a$ the element $r(a) \in R$ with $(a, r(a)) \in \alpha$. 
We define the mapping
\[
  \psi :    A \to A/\alpha \times Q
\] by
\[
\psi (a) = (\psi_1 (a), \psi_2 (a)) := (a /\alpha, \, a -_o r(a))
\]
for $a \in A$. Searching for its inverse, we define
\[
\varphi : A/\alpha \times Q \to A
\]
by $\varphi (a / \alpha, q) := q +_o r(a)$.
We will now prove that $\psi$ is bijective and that $\varphi = \psi^{-1}$.
To this end, we first compute
$\varphi (\psi (a)) = \varphi ( a/\alpha, a -_o r(a) ) = (a -_o r(a)) +_o r(a)$.
Since $L$ is a central series,  we have $[\alpha, 1_A] = 0_A$, and therefore
Lemma~\ref{lem:plus1}\eqref{it:p3}
yields $(a -_o r(a)) +_o r(a) = a$.
Second, we let $a \in A$ and $q \in Q$ and compute
\[
   \begin{split}
     \psi (\varphi (a/\alpha, q)) &= \psi (q +_o r(a)) \\
     &= \big( (q +_o r(a)) / \alpha, \, (q +_o r(a)) -_o r(q +_o r(a)) \big).
   \end{split}
\]
Since $q +_o r(a) \equiv_{\alpha} o +_o r(a) = r(a)$, we have
\[
   \begin{split}
     \big( (q +_o r(a)) / \alpha,  \, (q +_o r(a)) - r(q +_o r(a)) \big)
     &=
     \big( r(a) /\alpha, \, (q +_o r(a)) -_o r (a) \big).
   \end{split}
\]
 Now applying Lemma~\ref{lem:plus1}\eqref{it:p4}, we obtain
 that the last expression is equal to
 $(a / \alpha, q)$. 
 Thus $\psi$ and $\varphi$ are mutually inverse to each other, and
 hence bijective.
  Now we define the functions $+ : A \times A \to A$ and $-:A \to A$ by
\[
    \begin{array}{rcl}
      a + b & := &  \varphi \big( (a / \alpha) \oplus (b/\alpha), \, \psi_2 (a) +_o \psi_2 (b) \big) \text{ and} \\
      - b & := &  \varphi \big( \ominus (b/\alpha), \, -_o \, \psi_2 (b) \big).
     \end{array}  
\]
for $a, b \in A$.
We now prove that $+$ and $-$ satisfy the required properties and start with property~\eqref{it:f1}.
 We consider the algebra
 \[
     \algop{B}{\boxplus, \boxminus, o'} := \algop{A / \alpha}{\oplus,\ominus, o/\alpha } \times \algop{Q}{+_o,-_o, o}.
 \]
Since
\[
\begin{split}
   \psi (a + b ) &= \big( (a / \alpha) \oplus (b / \alpha),  \, \psi_2 (a) +_o \psi_2 (b) \big) \\
                   &= \big(  \psi_1 (a) \oplus \psi_1 (b),  \, \psi_2 (a) +_o \psi_2 (b) \big) \\
                   &= \big( \psi_1 (a), \psi_2 (a) \big) \boxplus \big(\psi_2 (b), \psi_2 (b) \big) \\
   &=  \psi (a) \boxplus \psi (b)
\end{split}
\]
for all $a, b \in A$,
and since, similarly, $\psi (-b) = \boxminus (\psi (b))$ and
$\psi (o) = o'$, the
mapping $\psi$ is an isomorphism from $\algop{A}{+,-,o}$ to $\algop{B}{\boxplus, \boxminus, o'}$ and
$\psi_1$ is an epimorphism from $\algop{A}{+,-,o}$ to $\algop{A / \alpha}{\oplus,\ominus, o/\alpha}$.
Since the kernel of $\psi_1$ is $\alpha$, we see that $\alpha$ is a congruence relation of
$\algop{A}{+,-,o}$, and therefore $+$ and $-$ preserve $\alpha_1$.
In order to show that $+$ and $-$ preserve $\alpha_i$ for $i \ge 2$, we let $i \in \{2,\ldots, m\}$
and observe that by the construction of $\oplus$ and $\ominus$ through the
induction hypothesis, $\alpha_i / \alpha$ is a congruence relation of $\algop{A/\alpha}{\oplus, \ominus, o/\alpha}$,
and therefore its pre-image $\beta$ under the homomorphism $\psi_1$,
given by
\[
\beta := \{ (a, b) \in A \mid \big(\psi_1 (a), \psi_1 (b)\big) \in \alpha_i / \alpha \},
\]
is a congruence
                 of $\algop{A}{+,-,o}$.
                 By its definition, $\beta= \{ (a, b) \in A \mid (a / \alpha, b / \alpha) \in \alpha_i / \alpha \}
                 = \alpha_i$. Hence, $+$ and $-$ preserve $\alpha_i$. This completes the proof of
                 item~\eqref{it:f1}.

                 Next, we prove item~\eqref{it:f3}.
                 The group $(Q, +_o,-_o, o)$ is isomorphic to the component $\ab{G}_1$ of the abelian group $\ab{G} := \prod_{i=1}^m \ab{G}_i$ associated with
                 $\ab{A}$, $L$, and $o$. The group associated with $\ab{A}/\alpha$, $L_1$ and $o/\alpha$ is
                 isomorphic to $\prod^m_{i=2} \ab{G}_i$.
                 Hence  $\algop{B}{\boxplus,\boxminus,o'} = \algop{A / \alpha}{\oplus,\ominus, o/\alpha} \times \algop{Q}{+_o,-_o, o}$
                 is isomorphic to $\ab{G}$, and therefore $\algop{A}{+,-,o}$ is isomorphic to $\ab{G}$.
                 This completes the proof of~\eqref{it:f3}, and thus item~\eqref{it:t1} of the statement
                 of the theorem is proved.

                 For~\eqref{it:f2}, we first consider the case $i \ge 2$.
                 Let $\ab{A'}$ be the expansion $\algop{A}{F \cup \{ +, - , o \}}$ of
                 $\ab{A}$. Its homomorphic image $\ab{A'}/\alpha$ is equal to
                 $\algop{A/\alpha}{F \cup \{ \oplus, \ominus, o/\alpha \}}$.
                 By the construction of $\oplus$ and $\ominus$ as functions preserving
                 the relations in~\eqref{eq:di} and the relational description of
                 centrality  (Lemma~\ref{lem:kiss}), we have
                 $[1_{A/\alpha}, \alpha_i/\alpha]_{\ab{A'} / \alpha} \le \alpha_{i-1}/\alpha$.
                 Hence 
                 $[1_A, \alpha_i]_{\ab{A'}} \le \alpha_{i-1}$, and thus the relational description
                 of commutators implies that $+$ and $-$ preserve $\gamma_i$.
                 Before proving that $+$ and $-$ also preserve $\gamma_1$, which encodes the centrality of $\alpha_1 = \alpha$ in $\ab{A'}$, we
                 prove some connections between $+, -, +_o, -_o$ and the Mal'cev term $d$.
                 First, we check that for all $a \in A$ and $q \in Q$, we have
                 \begin{equation} \label{eq:qprops}
                     q + a = q +_o a = a +_o q.
                 \end{equation}
                 For proving  the first equality, we compute
                 \[
                    \begin{split}
                      q + a &=  \varphi \big( ( q / \alpha) \oplus (a / \alpha ), \, \psi_2 (q) +_o \psi_2 (a) \big) \\
                               &=  \varphi \big( ( o / \alpha) \oplus  (a / \alpha ), \, (q -_o r (q)) +_o (a -_o r (a)) \big) \\
                      &=  \varphi \big(    a/\alpha, \, (q -_o o) +_o (a -_o r(a)) \big) \,\,
                               \\
                               &=  \varphi \big( a/\alpha, \, q +_o (a -_o r(a)) \big) \\
                               &=  \big(q +_o (a -_o r(a))\big) +_o r (a) \\
                               &=  q +_o \big( (a -_o r(a)) +_o r(a) \big) \,\, (\text{by Lemma~\ref{lem:plus1}\eqref{it:p7}})
                               \\
                               &=  q +_o a \,\,  (\text{by Lemma~\ref{lem:plus1}\eqref{it:p3}}).
                    \end{split}
                    \]
                   The second equality of~\eqref{eq:qprops} now follows from Lemma~\ref{lem:plus1}\eqref{it:p6}.
                 We will also need that for all $q \in Q$ and $a,b \in A$, we have
                 \begin{equation} \label{eq:dprops}
                   d (q +_o a, a, b) = q +_o b.
                 \end{equation}
                 This follows from Lemma~\ref{lem:plus1}\eqref{it:p8}.
                 Next, we observe that if $(v, w) \in \alpha$,
                 then $w -_o v \in Q$ and 
                 then $v = (w -_o v) +_o v$ by Lemma~\ref{lem:plus1}\eqref{it:p3}.

                 With these preparations, we are ready to prove that $+$ preserves $\gamma_1$.
                 To this end,
                 let $(x_1, x_2, x_3, x_4) \in \gamma_1$ and $(y_1, y_2, y_3, y_4) \in \gamma_1$.
                 We have to prove
                 \begin{equation} \label{eq:fgamma}
                 \big( x_1 + y_1, x_2 + y_2, x_3+ y_3, x_4 + y_4 \big) \in \gamma_1.
                 \end{equation}
                 By the fact that $+$ preserves $\alpha_1$, we obtain
                 $ (x_1 + y_1, x_2 + y_2) \in \alpha_1$. Hence for 
                 completing the proof of~\eqref{eq:fgamma}, we have to show
                 \[
                 d ( x_1 + y_1, x_2 + y_2, x_3 + y_3 ) =   x_4 + y_4.
                 \]
                 We set
                 \[
                    v := x_2, \, y := x_1 -_o x_2, \, w := y_2, \, z := y_1 -_o y_2.
                 \]
                 and compute
                 \[
                 \begin{split}
                    \mbox{} &
                  d \big( x_1 + y_1, x_2 + y_2, x_3 + y_3) \big) \\ &=
                  d \big( (y +_o v) +  (z +_o w), v + w, x_3 + y_3) \big) \\
                   &=
                  d \big( (y + v) + (z + w)) , v + w, x_3 + y_3 \big) \,\, (\text{by \eqref{eq:qprops}}) \\
                  &=
                  d \big(  (y + z) + (v + w), v + w, x_3 + y_3 \big)  \,\, (\text{because $\algop{A}{+}$ is abelian}) \\
                  &=
                  d \big( (y +_o z) +_o (v + w), v + w, x_3 + y_3 \big) (\text{by \eqref{eq:qprops}})  \\
                  &=
                  (y +_o z) +_o (x_3 + y_3) \,\, (\text{by~\eqref{eq:dprops}}) \\
                  &=
                  (y + z) + (x_3 + y_3) \,\, (\text{by~\eqref{eq:qprops}}) \\
                  &=
                  (y + x_3) + (z + y_3) \,\, (\text{because $\algop{A}{+}$ is abelian}) \\
                  &=
                  (y +_o x_3) +  (z +_o y_3) \,\, ( \text{by~\eqref{eq:qprops}}) \\
                  &=
                  d (y +_o v, v, x_3) + d (z +_o w, w, y_3)  \,\, (\text{by~\eqref{eq:dprops}}) \\
                  &=
                  d (x_1, x_2, x_3) +  d(y_1, y_2, y_3) \\ 
                  &=
                  x_4 + y_4 \,\, (\text{because $(x_1,x_2,x_3,x_4) \in \gamma_1$ and
                                                $(y_1,y_2,y_3,y_4) \in \gamma_1$}).                    
                     \end{split}
      \]
      This completes the proof of~\eqref{eq:fgamma}, and therefore $+$ preserves $\gamma_1$.
      We now show that $-$ preserves $\gamma_1$. As a first step, we show that for
      all $a,b,c \in A$ with $(a,b) \in \alpha$,
      we have
      \begin{equation} \label{eq:ab}
           a -_o b = a + (-b) \text{ and } d (a,b,c) = a + (-b) + c.
      \end{equation}
      From Lemma~\ref{lem:plus1}\eqref{it:p3},
      we obtain $a = (a -_o b) +_o b$, which is equal to
      $(a -_o b) + b$ by~\eqref{eq:qprops}.
      Since $\algop{A}{+,-,o}$ is a group, we
      therefore have
      $a + (-b) = a -_o b$, establishing the first part of~\eqref{eq:ab}.
      For the second part, we observe that
      \[
      \begin{split}
        d (a,b,c) &=
        d (a -_o b) +_o b, b, c) \\
        &=
        (a -_o b) +_o c \,\, (\text{by \eqref{eq:dprops}}) \\
        &=
        (a -_o b) + c \,\, (\text{by \eqref{eq:qprops}}) \\
        &=   
          (a + (-b)) + c \,\, (\text{by the first part of \eqref{eq:ab}}).
      \end{split}
      \]
      We now take $(x_1,x_2,x_3,x_4) \in \gamma_1$, and prove that
      $(-x_1, -x_2, -x_3, -x_4) \in \gamma_1$. Since $-$ preserves $\alpha$,
      $(-x_1, -x_2) \in \alpha$, and thus it remains to show that
      \begin{equation} \label{eq:mp}
        d(-x_1, -x_2, -x_3) = -x_4.
      \end{equation}
      We have
      \[
      \begin{split}
        d (-x_1, -x_2, -x_3) &=
        (-x_1) + x_2 + (-x_3) \,\, (\text{by~\eqref{eq:ab}}) \\
        &=
        - (x_1 +  (- x_2) + x_3) \,\, (\text{because $\algop{A}{+,-,o}$ is an abelian group}) \\
        &=
        - d(x_1, x_2, x_3) \,\, (\text{by~\eqref{eq:ab}}) \\
          &=
          - x_4.
      \end{split}
      \]
      Hence $(-x_1,-x_2,-x_3,-x_4) \in \gamma_1$, and therefore
      $-$ preserves $\gamma_1$, which completes the proof of~\eqref{it:f3}.

      Now to establish~\eqref{it:t2}, we observe that for each $i \in \{1,\ldots, m\}$,
      the nilpotency of $\ab{A}$ implies 
      $[1, \alpha_i]_{\ab{A}} \le \alpha_{i-1}$.
      Thus (by Lemma~\ref{lem:kiss}) each fundamental operation of $\ab{A}$ preserves $\gamma_i$. Since also $+$ and $-$ preserve
      $\gamma_i$ by item~\eqref{it:f2}, all fundamental operations of $\ab{A'}$ preserve $\gamma_i$,
      which implies $[1, \alpha_i]_{\ab{A'}} \le \alpha_{i-1}$. Hence $\ab{A'}$ is nilpotent
      of class at most $m$.  \qed

  \section{Clones of polynomials} \label{sec:cop}
  All finitary functions on a finite field are induced by polynomials.
  When considering polynomials instead of functions, we can use
  notions such as \emph{degree} or \emph{monomial}. Such an
  approach has been used, e.g., in \cite{Kr:CFSO}.
    In this section, we will study polynomials
    in the polynomial ring $\KX = \bigcup_{n \in \N} \Kxn{n}$ over
    countably many variables over some (not necessarily finite) field $\ab{K}$.
    Adapting \cite{CF:FCAR}, we define the product of $A, B \subseteq \KX$
    by
    \[
        A B = \{ p (q_1, \ldots, q_n) \setsuchthat
        n \in \N, p \in A \cap \Kxn{n}, q_1, \ldots, q_n \in B \}.
    \]    
    Here $p(q_1, \ldots, q_n)$ denotes the polynomial obtained from
    $p$ by substituting simultaneously each variable $x_i$ with
    $q_i$.
    We say that a subset $C$ of $\KX$ is a \emph{clone of polynomials} if
    for each $i \in \N$, $x_i \in C$ and $CC \subseteq C$.
    Given a subset $F$ of $\KX$, we use $\Clop (F)$ to
    denote the clone of polynomials that is generated by $F$.
    By $L$, we denote the set
    \[
    L := \{ \sum_{i = 1}^n a_i x_i \setsuchthat
    n \in \N,  \forall i \in \{1,\ldots, n\} : a_i \in \Z \}
    \]
    Hence $L = \Clop (\{x_1+x_2,-x_1,0\})$, and if $F$ is a nonempty subset of
    $\KX$, then
    $L F$ is exactly the subgroup of $\algop{\KX}{+,-,0}$ generated by $F$.

    We notice that a clone of polynomials is not a clone in the usual sense,
    since its elements are polynomials, and not finitary functions on some set.
    A bridge between these concepts is provided in the following lemma.
    For a field $\ab{K}$, $f \in \Kxn{n}$ and $m \ge n$, we let $f^{\ab{K},m}$ be the $m$-ary function
    that $f$ induces on $K$. For example $x_3^{\ab{K}, 5}$ induces the projection
    $(a_1, \ldots, a_5) \mapsto a_3$ on $K$. 
    \begin{lem}
      Let $\ab{K}$ be a field, and let $C \subseteq \KX$ be a clone of polynomials.
      Let
      \[
      C' := \{ f^{\ab{K}, m} \mid n \in \N, f \in \Kxn{n}, m \ge n \}.
      \]
      Then $C'$ is a clone on the set $K$.
    \end{lem}
    Given this close connection, it is not suprising that we may transfer some
    results from clone theory to clones of polynomials.
    \begin{lem}[Associativity Lemma, cf. \cite{CF:FCAR}] \label{lem:ass}
      Let $A,B,C \subseteq \KX$, let $L := \Clop (\{x_1 + x_2, -x_1 , 0\})$,
      and let $P := \Clop (\emptyset) = \{x_i \mid i \in \N\}$.
      Then we have:
      \begin{enumerate}
      \item \label{it:a1} $(AB)C \subseteq A (BC) \subseteq (A (BP)) C$. In particular if $BP \subseteq B$, then
           $(AB)C= A(BC)$.
      \item \label{it:a2} $L (AL)$ is closed under composition with polynomials in $L$ from both sides;
        in other words,
             $L \, (L (AL)) \subseteq L (AL)$ and
             $(L (AL)) \, L \subseteq L (AL)$.
      \end{enumerate}
    \end{lem}
    \emph{Proof:}
    The proof of item~\eqref{it:a1} is straightforward and can be developed
    along the lines of~\cite{CF:FCAR}.
    Item~\eqref{it:a2} then follows by observing
    $L(L (AL)) \subseteq (L (LP)) (AL) = (LL) (AL) = L (AL)$  
     and
    $(L (AL))L \subseteq L ((AL) L) \subseteq L (A (LL)) = L (AL)$. \qed

     For a set $C \subseteq \KX$, let  $C^{(0)} := \{x_i \mid i \in \N\}$ and
     for $n \in \N_0$, $C^{(n+1)} = C^{(n)} \cup (C \, C^{(n)})$. For $f \in \KX$, the \emph{depth of $f$
       with respect to $C$}, denoted by $\delta_C (f)$,  is the smallest $n \in \N$ with $f \in C^{(n)}$, and undefined
      if no such $n$ exists.
      \begin{lem} \label{lem:clonegen}
        Let $\ab{K}$  be a field, and let $C \subseteq \KX$. Then
        the clone generated by $C$, $\Clop (C)$, is equal to $\bigcup \{ C^{(n)} \mid n \in \N_0 \}$.
        If $M \subseteq \KX$ is such that $\{x_i \mid i \in \N\} \subseteq M$
        and $CM \subseteq M$, then $\Clop (C) \subseteq M$.
      \end{lem}
      \emph{Proof:} Let $U := \bigcup \{ C^{(n)} \mid n \in \N_0 \}$. Then
      $C \subseteq C^{(1)} \subseteq U \subseteq \Clop (C)$, and hence it is
      sufficient to prove $UU \subseteq U$. To this end, we prove by induction on
      $n$ that $C^{(n)} U \subseteq U$.
      This is obvious for $n=0$. For the induction step, we let $n \in \N_0$, $m \in \N$,
      $u \in (C^{(n+1)} \setminus C^{(n)}) \cap  \Kxn{m}$, and $v_1, \ldots, v_m \in U$.
      Since $u \in CC^{(n)}$, there are $l \in \N$, $r \in C$ and $s_1, \ldots, s_l \in C^{n}$
      with $u = r (s_1, \dots, s_l)$.
      Now $u (v_1, \ldots, v_m) = r (s_1 (v_1, \ldots, v_m), \ldots, s_l (v_1, \ldots, v_m))$.
      Then each $s_i (v_1, \ldots, v_m)$ is an element of $U$ by the induction hypothesis.
      Thus there is $k \in \N$ such that
      $\{ s_i (v_1, \ldots, v_m) \mid i \in \{1,\ldots, l\} \}
      \subseteq C^{(k)}$, and therefore $u (v_1, \ldots, v_m) \in C^{(k+1)} \subseteq U$. This completes
      the induction step; therefore $UU \subseteq U$ and $U = \Clop (C)$.
      We show the second part by proving that for all $n \in \N$, $C^{(n)} \subseteq M$.
      The induction basis $n=0$ follows from the condition $\{x_i \mid i \in \N\} \subseteq M$.
      For the induction step, let $n \in \N_0$.
      Then $C^{(n+1)} =  C^{(n)} \cup (CC^{(n)}) \subseteq
            M \cup CM$
      by the induction hypothesis. Applying
      the assumption $CM \subseteq M$, we obtain $C^{(n+1)} \subseteq M$.
      Thus $\Clop (C) \subseteq M$. \qed

    The \emph{total degree of of a monomial} is defined by
    \[
    \deg ( a \prod_{i = 1}^n x_i^{\alpha_i} ) :=
    \sum_{i=1}^n \alpha_i
    \]
    for $n \in \N$ and $a \in K \setminus \{0\}$ and the \emph{total degree
    of a polynomial} is the maximum of the total degrees of
    its monomials.
    A polynomial is called \emph{homovariate} if all of its
    monomials contain exactly the same variables.
    For example, over $K = \Z_7$, each of the polynomials
    $5 x_1 x_2^3 x_4 - 2 x_1^{17} x_2 x_4^{3} + x_1^{6} x_2^3 x_4^{20}$,
    $x_2 + 6 x_2^4$ and $2$ is homovariate, but none of 
    the polynomials $x_1 + x_2$, $1 +3 x_1^3 + x_1^5$ is homovariate. 
    For a finite subset $I$ of $\N$, the \emph{homovariate component}
    $H_I (p)$ of the polynomial $p$ with respect to $I$ is defined as
    the sum
    those monomials whose set of variables is $I$.
    As an example, we compute
    \[
    H_{\{2,3,4\}} (5 x_2^2 x_3 + 7 x_2^2 x_3 x_4^5 + x_1x_2x_3x_4 +
    4 x_4^3 x_5 + 13 x_2^6 x_3^8 x_4^7)
    = 
    7  x_2^2 x_3 x_4^5 +
    13 x_2^6 x_3^8 x_4^7
    \]
    and $H_{\{2,3\}} (x_2^2 + x_3) = 0$.
    Hence each polynomial is the sum of all of its homovariate components.
    For a set of  polynomials $F \subseteq \KX$,
    \[
       \Hoc(F) := \{ H_I (f) \mid
       I \subseteq \N, f \in F\}
       \]
       is the set of the homovariate components of elements of $F$.
       We note that by this definition, for every polynomial $f$,
       $0 \in \Hoc (\{f\})$: let $j \in \N$ be such that
       $x_j$ does not
       occur in $f$.  Then $H_{\{j\}} (f) = 0$. We also see that
       for every
       $f \neq 0$, the set $\Hoc (\{ f \})$ has at least two
       elements.
    \begin{lem}  \label{lem:hv}
      Let $F \subseteq \KX$, and let $L := \Clop (\{x_1+x_2, -x_1, 0\})$.
      Then we have:
      \begin{enumerate}
          \item $F \subseteq L \, \Hoc (F)$.   
          \item $\Hoc (F) \subseteq L (FL)$.
      \end{enumerate}      
        \end{lem}  
        \emph{Proof:}
    For every $f \in F$, we have
    $f = \sum_{h \in \Hoc (f)} h$, and therefore
    $F \subseteq L \, \Hoc (F)$.
    For proving the second assertion,
    we show that for every $n \in \N$
     and every $f \in \Kxn{n}$,
     every $h \in \Hoc(\{f\})$ satisfies $h \in L (\{f\} L)$.
     We proceed by induction on the  number of homogeneous components
     of $f$, i.e., on $|\Hoc (\{f\})|$. If $|\Hoc(\{f\})|  = 1$,
     then $f = 0$, therefore $\Hoc(\{f\}) = \{0\}$ and thus
     $\Hoc (\{f\}) \subseteq L(\{f\} L)$.
     For the induction step, we assume that
     $|\Hoc (\{ f \})| \ge 2$. We list
     all subsets of $\{1,\ldots, n\}$ as $(I_1, I_2, \ldots, I_{2^n})$
     in such a way that for all $i, j \in \{1,2,\ldots, n\}$, we have
     $I_i \subseteq I_j \Rightarrow i \le j$.
     Now
     \[
     f = \sum_{j = 1}^{2^n} H_{I_j} (f).
     \]
     Let $k$ be minimal with $H_{I_k} (f) \neq 0$.
     Then, for all $j$ with $j>k$, $I_j \not\subseteq I_k$,
     and hence there  is $m \in I_j$ such that $m \not\in I_k$.
     We produce $f'$ from $f$ by setting all variables whose indices
     are not in $I_k$ to $0$. Clearly $f' \in \{f\} L$.
     Since all summands of $f$  for $j > k$ become $0$ by
     this setting, we have
     $f' = H_{I_k} (f)$.
     By the induction hypothesis, every homogeneous component
     of $f - f' = \sum_{j = k+1}^{2^n} H_{I_j} (f)$ lies in
     $L ( \{f - f'\} L)$. Since $f' \in \{ f \} L$, we have
     $f - f' \in L (\{f\} L)$, and therefore
     $L ( \{f - f'\} L) \subseteq L ( (L (\{f\} L)) L) = L (\{f\} L)$ and so we obtain that
      $\{ H_{I_{k+1}}, \ldots, H_{I_{2^n}} \} \subseteq
     L (\{f\} L)$. \qed

     The following theorem will help us to represent term functions of the algebra
     $\ab{A}$ as sums of absorbing functions. Informally, the idea is the following:
     Suppose that we have a universal algebra $\ab{A} = \algop{A}{+,-,0, (f_i)_{i \in I}}$,
     and let $F := \{f_i \mid i \in I\}$. To simplify the discussion, we assume
     that all $f_i$ have positive arity.
     Every  term function of $\ab{A}$ can be represented
     as by a tree whose leaves are variables or $0$, and whose other nodes are elements
     of $F \cup \{+,-\}$. Our goal is to move $+$ and $-$ to the top of the tree.
     To this end, we transform the tree into a tree
     whose nodes are labelled by a new set of functions, $H$, and by $+$ and $-$.
     All functions in $H$ will be absorbing, and in the new tree,
     no node labelled by $+$ or $-$ will appear inside a subtree rooted by
     an element of $H$.
     Deviating from this explanation, we will not work with the operations of the algebra
     $\ab{A}$ directly, but rather with polynomials over a field whose universe
     is $A$. Given a set $F$ of polynomials,
     we will obtain a set $H$ of homovariate polynomials such that
     each polynomial in $\Clop (F \cup \{ x_1+x_2,-x_1,0 \})$ is
     a sum of compositions of polynomials in $H$; this set of
     sums of compositions is the just the product $L \, C$, where
     $C = \Clop (H)$.
             
       \begin{thm} \label{thm:LClo}
   Let $\ab{K}$ be a field, let $F \subseteq \KX$, $L := \Clop (\{x_1+x_2,-x_1, 0\})$,
   and let $n \in \N$ be such
   that the total degree of each $f \in F$ is at most $n$.
   Then there exists a set $H \subseteq \Kxn{n}$ of homovariate polynomials
   such that
   \begin{equation} \label{eq:lhfclaim}
   L \,  \Clop (H) = \Clop (F \cup \{ x_1+x_2,-x_1,0 \})
   \end{equation}
   and the
   total degree of each $h \in H$ is at most $n$.
 \end{thm}
 \emph{Proof:}
 In the case $F = \varnothing$, we choose $H := F$ obtain that both sides
 are equal to the subgroup of $\algop{\KX}{+,-,0}$ generated by $\{x_i \mid i \in \N\}$.
 Let us now assume $F \neq \varnothing$.
 We consider the subgroup
 \(
 S = \big( L(FL) \big) \cap \Kxn{n} 
 \)
 of
 $\algop{\Kxn{n}}{+,-,0}$.
 Let
 \[
     H := \Hoc (S) = \Hoc \Big( \big( L(FL) \big) \cap \Kxn{n} \Big). 
 \]
 Now each $h \in H$ is a homovariate component of a polynomial in $L (FL)$, and
 has therefore total degree at most~$n$. We now  start to establish~\eqref{eq:lhfclaim}. 
 By Lemma~\ref{lem:hv},
 we have $H = \Hoc (S) \subseteq L (SL)$. Since $S \subseteq L (FL)$,
 and since by Lemma~\ref{lem:ass}\eqref{it:a2}, $L(FL)$ is closed under composition with polynomials from $L$
 from both sides, we obtain $L (SL) \subseteq L(FL)$, and therefore
 $H \subseteq L (FL)$, and then also 
 \begin{equation} \label{eq:hllfl}
 HL \subseteq L (FL).
 \end{equation}
 We will now prove
 \begin{equation} \label{eq:gh}
 \Clop (H \cup \{ x_1+x_2, -x_1, 0 \})
 =
 L \, \Clop (H).
 \end{equation}
 $\supseteq$: Both sets $L$ and $\Clop (H)$ are subsets of
 $\Clop (H \cup \{x_1 + x_2 ,- x_1, 0 \})$. Since $\Clop (H)$ is a clone,
 their product $L \, \Clop (H)$ is also a subset of $\Clop (H)$.
 
 $\subseteq$: We use Lemma~\ref{lem:clonegen}
 with $C := H \cup \{ x_1 + x_2, - x_1, 0\}$ and $M = L \, \Clop(H)$, and observe
 that $\{x_i \mid i \in \N\} \subseteq M$. For proving
 $CM \subseteq M$, we observe that using the Associativity Lemma (Lemma~\ref{lem:ass}),
 we obtain
 $CM = HM \cup \{ x_1 + x_2, - x_1, 0\} M \subseteq
  HM \cup L (L \, \Clop (H)) = HM \cup L \, \Clop (H) = HM \cup M$.
 Hence what remains to prove is $HM \subseteq M$.
 To this end, we will show that
 for all  $t_1, \ldots, t_{n} \in L \, \Clop (H)$ and 
 for all  $g \in H$, we have
 \begin{equation} \label{eq:gti}
        g (t_1, \ldots, t_{n}) \in L \, \Clop (H).
 \end{equation}
 We fix $t_1, \ldots, t_{n} \in L\, \Clop (H)$ and $g \in H$. 
 Each $t_i$ is a sum of elements in $\Clop (H) \cup \{ -p \mid p \in \Clop (H) \}$.
 We collect these summands and thereby find $N \in \N_0$,
 $s_1, \ldots, s_N \in \Clop (H)$, $\sigma :\{1,2,\ldots, N\} \to \{0,1\}$, and
 $(m_i)_{i=1}^n$ with $0 = m_0 \le m_1 \le m_2 \le \dots \le m_n = N$ such that
 for each $i \in \{1,2,\ldots, n\}$, we have
 \[
 t_i = \sum_{j=m_{i-1} + 1}^{m_i} (-1)^{\sigma(j)} \, s_j.
   \]
   We define $e \in \Kxn{N}$ by
   \[
   e(x_1, \ldots, x_N) := g ( \sum_{j=1}^{m_1} (-1)^{\sigma(j)} \, x_j, \ldots,
   \sum_{j = m_{n-1} + 1}^{N}  (-1)^{\sigma(j)} \, x_j),
   \]
   which implies
   \[
   e(s_1, \ldots, s_N) = g (t_1, \ldots, t_n).
   \]
     Then $e \in H L$, and thus by~\eqref{eq:hllfl}, $e \in L (FL)$.     
  We decompose $e$ into its homovariate components and obtain
     \begin{equation} \label{eq:edec}
        e = \sum_{I \subseteq \{1, \ldots, N\}} H_I (e).
     \end{equation}
        Let $I \subseteq \{1, \ldots, N\}$.
        We first observe that $H_I (e) \in \Hoc (L (FL))$, which
        by Lemma~\ref{lem:hv} is a subset of $L \Big( \big( L(FL) \big) L \Big)$.
        Hence by Lemma~\ref{lem:ass}\eqref{it:a2}, we obtain
        \begin{equation} \label{eq:hielfl}
          H_I (e) \in L (FL).
        \end{equation}  
        We will now show that for each $I \subseteq \{1,2,\ldots, N\}$,
        we have
        \begin{equation} \label{eq:hie}
          H_I (e) \in \{0\} \cup \Clop (H) \subseteq L \, \Clop (H).
        \end{equation}
  We first consider the case $|I| > n$.        
  Since $e \in L(FL)$ is obtained by adding and substituting linear polynomials into polynomials
  from $F$, 
  $e$  has total degree at most $n$.
  Hence $H_I (e) = 0$. 
  In the case $|I| \le n$, we let $\pi :\{1,2,\ldots, N\} \to \{1,2,\ldots, N\}$ be
  a bijection such that $I \subseteq \pi [\{1,2,\ldots, n\}]$. Then clearly
  $\pi^{-1} [I] \subseteq \{1,2,\ldots, n\}$. 
       We define
     \begin{equation*} \label{eq:subst}
          p_I (x_1, \ldots, x_N) := H_I (e) \, ( x_{\pi^{-1} (1)}, x_{\pi^{-1} (2)}, \ldots, x_{\pi^{-1} (N)} ).
     \end{equation*}
     Then by~\eqref{eq:hielfl}, $p_I \in \big(L(FL)\big) L \subseteq L (FL)$. 
     Since $H_I (e) \, (x_1, \ldots, x_N)$  contains only variables $x_i$ with $i \in I$,
     $H_I (e)\,  (x_{\pi^{-1} (1)}, \ldots, x_{\pi^{-1} (N)})$ contains only
     $x_{\pi^{-1} (i)}$ with $i \in I$.  Thus $p_I \in K [ x_{\pi^{-1} (i)} | i \in I ]$,
     which implies $p_I \in \Kxn{n}$.
      Therefore,  $p_I$ is a homovariate polynomial in $(L (FL)) \cap \Kxn{n}$, and
     thus $p_I \in H$. 
     Now we compute
     \[
        \begin{split} 
          p_I (s_{\pi (1)}, \ldots, s_{\pi (N)}) &=
          H_I (e) \, ( s_{\pi (\pi^{-1} (1))}, \ldots, s_{\pi (\pi^{-1} (N))})  \\
          &=
          H_I (e) \, ( s_1, \ldots, s_N).
        \end{split}
      \]
      Since $p_I \in H$, we have $p_I (s_{\pi (1)}, \ldots, s_{\pi (N)}) \in \Clop (H)$.
      Therefore, $H_I (e) \, (s_1, \ldots, s_N) \in \Clop (H)$, which
      completes the proof of~\eqref{eq:hie}. 
      Using~\eqref{eq:edec}, we obtain
      that $e (s_1, \ldots, s_N) \in L \,\Clop (H)$. This completes the proof
      of~\eqref{eq:gti}. Now applying Lemma~\ref{lem:clonegen} we obtain the ``$\subseteq$''-inclusion
      of \eqref{eq:gh}. 

      We finish the proof by establishing that
      \begin{equation} \label{eq:fh}
      \Clop (F \cup \{x_1+x_2,-x_1,0\}) = \Clop (H \cup \{x_1+x_2,-x_1,0 \}).
      \end{equation}
      For $\subseteq$, we first observe that $F \subseteq L\, \Hoc (F)$.
      Each $g \in \Hoc(F)$ contains at most $n$ variables. Replacing
      these $n$ variables by $x_1, \ldots, x_n$ and undoing this replacement
      afterwards, we obtain
      \begin{equation} \label{eq:hfl}
      \Hoc (F) \subseteq  \big( (\Hoc(F) \, L) \cap \Kxn{n} \big) \, L.
      \end{equation}
      The next goal is to prove
      \begin{equation} \label{eq:hfl1}
        (\Hoc(F) \, L) \cap \Kxn{n}  \subseteq LH.
       \end{equation}  
      By Lemma~\ref{lem:hv},
      $\Hoc (F) \subseteq  L(FL)$, thus $\Hoc (F) \, L \subseteq L (FL)$.
      Therefore
      \[
      (\Hoc(F) \, L) \cap \Kxn{n}  \subseteq (L (FL)) \cap \Kxn{n}.
      \]
      Now by Lemma~\ref{lem:hv},
      \[
          (L (FL)) \cap \Kxn{n} \subseteq L \, \Hoc \big(L (FL) \cap \Kxn{n}\big) = LH,
          \]
    which
      completes the proof of~\eqref{eq:hfl1}.
      Combining~\eqref{eq:hfl} and~\eqref{eq:hfl1}, we get
      \[
      \begin{split}
        F &\subseteq L \, \Hoc (F) \\
        &\subseteq L \, \Big( \big((\Hoc (F) \, L) \cap \Kxn{n} \big) L \Big) \,\,(\text{by~\eqref{eq:hfl}}) \\
        &\subseteq L \, (
                   (LH)
          L  )  \,\, (\text{by~\eqref{eq:hfl1}}).
      \end{split}
      \]
      Since  both $L$ and $H$ are subsets of $\Clop(H \cup \{x_1 + x_2, -x_1, 0\})$,
      we obtain $F \subseteq \Clop(H \cup \{x_1 + x_2, -x_1, 0\})$. From this,
      the inclusion $\subseteq$ of \eqref{eq:fh} immediately follows.
      For the other inclusion in~\eqref{eq:fh}, we use~\eqref{eq:hllfl} to obtain
      $H \subseteq L(FL)$. Hence $H \subseteq  \Clop (F \cup \{x_1+x_2,-x_1,0\})$.
      This proves~\eqref{eq:fh}; together with~\eqref{eq:gh}, this establishes the claim
      in~\eqref{eq:lhfclaim}. \qed

      \section{Clones of finitary functions}
         We call a finite algebra $\ab{A} =
      \algop{A}{+,-,0, (f_i)_{i\in I}}$ an
      \emph{expanded elementary abelian group} if $\algop{A}{+,-,0}$ is
      a finite abelian group of prime exponent.
            We call $*$ a \emph{field multiplication} on $\ab{A}$
      if $\ab{K} := \algop{A}{+,-,0,*}$ is a field; $\ab{K}$ is
      then a \emph{field associated with $\ab{A}$}.
      We do not claim that such a multiplication has any further connection
      to the algebra $\ab{A}$.
      \begin{lem} \label{lem:abs1}
        Let $\ab{K}$ be a field,
        let $n \in \N$, and let $p \in \Kxn{n}$ be such that
        $p^{\ab{K},n}$ is an absorbing function from $K^n$ to $K$.
        Then $p^{\ab{K},n} = (H_{\{1,2,\ldots, n\}} (p)) ^{\ab{K},n}$.
      \end{lem}
      \emph{Proof:}
      We proceed by induction on the number $k :=
      \# \{ I \subseteq \{1,2,\ldots, n\} \mid H_I (p) \neq 0 \}$ of
      non-zero homovariate components of $p$.
      If $k = 0$, then $p = 0$ and $H_{\{1,2,\ldots, n\}} = 0$.
      If $k \ge 1$, we let $I$ be minimal with respect to $\subseteq$
      such that $H_I(p) \neq 0$. If $I = \{1,2,\ldots,n\}$, then
      $p = H_I (p)$. If $I \not= \{1,2,\ldots, n\}$, we write $p$
      as the sum of its homovariate components, which means
      \[
      p = \sum_{J \subseteq \{1,2,\ldots,n\}} H_J (p).
      \]
      We set all $x_i$ with
      $i \not\in I$ to $0$ and obtain
      $0 = H_I (p)^{\ab{K}, n}$. Therefore, $q := p - H_I (p)$ satisfies
      $p^{\ab{K},n} = q^{\ab{K},n}$. By the induction hypothesis $q^{\ab{K},n} =
      H_{\{1,2,\ldots, n\}} (q) ^{\ab{K},n}$.
      Now since $H_{\{1,2,\ldots, n\}} (q) = H_{\{1,2,\ldots, n\}} (p) $, we obtain
      $(H_{\{1,2,\ldots, n\}} (q))^{\ab{K},n} = (H_{\{1,2,\ldots, n\}} (p))^{\ab{K},n}$. \qed
       \begin{thm} \label{thm:boundabs}
        Let $\ab{A} =  \algop{A}{+,-,0, (f_i)_{i\in I}}$ an expanded elementary abelian
        group, and let $m, k \in \N$.
        We assume that for each $i \in I$, the
         arity of $f_i$ is at most $m$, and that
        $\ab{A}$ is nilpotent of class at most $k$.
         Then all absorbing polynomial functions of $\ab{A}$ are of
         essential arity at most $(m (|A| - 1))^{k-1}$.
      \end{thm}
       \begin{proof}
         If $|A| = 1$, then all polynomial functions are of essential
         arity~$0$, and hence the claim holds. We will now assume
         $|A| > 1$.
         We let $(\alpha_i)_{i \in \N_0}$ be the lower central series
         of $\ab{A}$ defined by $\alpha_0 := 1_A$ and $\alpha_i = [1_A, \alpha_{i-1}]$
         for $i \in \N$, and for $i \in \N_0$, we define $A_i := 0/\alpha_i$ to be the ideal
         of $\ab{A}$ associated with $\alpha_i$; hence $A_0 = A$.
         Then by $k$-nilpotency,
        $A_k = 0$.  
        Let $\ab{K}$ be a field associated with $\ab{A}$.
        For each $i \in I$, we let $m_i$ be the arity of $f_i$, and
        we choose
        ${f'_i} \in \Kxn{m}$ such that
        \[
        {(f'_i)}^{\ab{K},n} (a_1, \ldots, a_m) = f_i (a_1, \ldots, a_{m_i})
        \]
        for all $a_1, \ldots, a_m \in A$ and
        $\deg_{x_j} ({f'_i}) < |A|$ for all $j \in \{1,2,\ldots, m\}$.
     
        Then the total degree of ${f'_i}$ is at most $n := m (|A| - 1)$.
        Let
        \[
            F := \{ {f'_i} \mid i \in I \}.
        \]
            Then $F \subseteq \Kxn{m}$.
            We use Theorem~\ref{thm:LClo} to obtain
            a set $H \subseteq \Kxn{n}$ of homovariate polynomials
   such that
   \begin{equation} \label{eq:U}
   L \, \Clop (H) = \Clop (F \cup \{ x_1+x_2,-x_1,0 \})
   \end{equation}
   and the
   total degree of each $h \in H$ is at most $n$.
   
   We will show next that
   for all  $l \in \N$, for all $N \in \N$, and
   for all $p \in \Clop (H) \cap \Kxn{N}$,
   the following property holds:
   \begin{equation}  \label{eq:pnl}
     \text{if $p$ contains at least $n^{l-1} + 1$ variables,} \\
            \text{ then
              $p^{\ab{K},N} (A^N) \subseteq A_l$.}
   \end{equation}         
            Seeking a contradiction, we let $l \in \N$ be minimal
            such that there is an $N \in \N$ and a  $p \in \Clop (H) \cap \Kxn{N}$ that contains at least
            $n^{l-1} + 1$ variables and $p^{\ab{K},N} (A^N) \not\subseteq A_l$.
            Among those $p$, we choose one of minimal depth $\delta_H (p)$ (as defined
            before Lemma~\ref{lem:clonegen})
            with respect to $H$.
            Since $l \in \N$, $p$ contains at least two variables,
            and thus $p$ is not a variable and not a constant polynomial.
            Therefore, \[ p = h (t_1, \dots, t_n) \] with $h \in H$ nonconstant
            and $t_1, \ldots, t_n \in \Clop (H)$.
            
            In the case that $h$ contains only one variable, we let
            $x_j$ be this variable. The polynomial $t_j$ must then
            also contain at least $n^{l-1} + 1$ variables. By the
            minimality of $\delta_H (p)$, $t_j^{\ab{K},n} (A^n) \subseteq A_l$.
            Since $h$ is homovariate, the function
            $g_1: a_j \mapsto h^{\ab{K},n} (a_1, \ldots, a_n)$, which is
            formally defined
            by
            \[
                   g_1 = \{ (a_j, h^{\ab{K},n} (a_1, \ldots, a_n))
                            \mid
                            a_1, \ldots, a_n \in A \},
            \] satisfies $g_1 (0) = 0$. Since $h^{\ab{K},n}$ is a term operation of
            $\ab{A}$, we have $g_1 (A_l) \subseteq A_l$, and therefore
            $p^{\ab{K},N} (A^N) \subseteq A_l$, contradicting the choice of $p$.

            In the case that $h$ contains exactly $r$ variables
            with $2 \le r \le n$, we let $x_{j_1}, \ldots, x_{j_r}$
            be these variables, and define $g_2 : A^r \to A$ by
            \begin{equation*} %
              g_2: (a_{j_1}, a_{j_2}, \ldots, a_{j_r}) \mapsto
              h^{\ab{K},n} (a_1, \ldots, a_n)
             \end{equation*}
            We first show that for all $i_1, \ldots, i_r \in \N_0$,
            \begin{equation}  \label{eq:g2abs}
              g_2 (A_{i_1} \times \cdots \times A_{i_r}) \subseteq A_{\max (\{i_1, \ldots, i_r\}) +1}.
            \end{equation}  
            To this end, we fix $(a_1, \ldots, a_r) \in \prod_{s=1}^r A_{i_s}$.
            The function $g_2$ is a  term function of $\ab{A}$.
            Let $u$ be such that
            $i_u = \max (\{i_1, \ldots, i_r\})$,and let $v \in \{1,\ldots, r\} \setminus \{u\}$.
            We define
            \[g_3 (x,y) := g_2 (a_1, \ldots, a_{u-1}, x, a_{u+1}, \ldots, a_{v-1}, y, a_{v+1}, \ldots, a_r). \]
            Then $g_3$ is a polynomial function of $\ab{A}$. Since $h$ is homovariate,
            $g_3 (a,0) = g_3 (0, a) = 0$ for all $a \in A$. Denoting $0/[\alpha_{d}, \alpha_{e}]$ simply
            by $[A_d, A_e]$, Lemma~\ref{lem:cideals}
            implies $g_3(a_u, a_v) \in [A_{i_u}, A_{i_v}] \subseteq [A_{i_u}, A] \subseteq A_{i_u + 1}$.
            This completes the proof of~\eqref{eq:g2abs}.
            
            Continuing with the proof of~\eqref{eq:pnl},
            we first consider the case $l = 1$.
            Then by~\eqref{eq:g2abs}, $g_2 (A^r) = g_2 ({A_0}^r) \subseteq A_1$, and therefore
            $h^{\ab{K},n} (A^n) \subseteq A_1$. Hence $p^{\ab{K}, N} (A^N) \subseteq A_1$,
            contradicting the choice of $p$.
            
             In the case $l \ge 2$, one of the polynomials
              $t_{j_1}, \ldots, t_{j_r}$ contains at least
              $n^{l-2} + 1$ variables: if all contained at most
              $n^{l-2}$ variables, also $p = h (t_1, \ldots, t_n)$
              would contain at most $r n^{l-2} \le n^{l-1}$ variables,
              contradicting the choice of $p$.
              Let $s \in \N$ be such that $t_{j_s}$ contains at
              least $n^{l-2} + 1$ variables. By the minimality of $l$,
              we see that $t_{j_s} (A^n) \subseteq A_{l-1}$.
              By~\eqref{eq:g2abs},
                    $g_2 (A \times \cdots \times A \times  A_{l-1} \times A \times \cdots \times A) \subseteq
               A_l$, where $A_{l-1}$
               occurs at place $s$. 
              Thus $p^{\ab{K}, N}  (A^N) \subseteq A_l$,
              contradicting again the choice of $p$.
              This completes the proof of~\eqref{eq:pnl}.

              Setting $l := k$, we see that every $p \in \Clop (H)$ that contains at least $n^{k-1} + 1$
              variables induces the constant $0$ function on $\ab{K}$.

              We will now show that all absorbing polynomial functions
              of $\ab{A}$
              depend on at most $n^{k-1}$ variables.
              To this end, let $N > n^{k-1}$, and let $q$ be an $N$-ary absorbing polynomial function of
              $\ab{A}$.
              Then there is $M \in \N$ and there are
              $t \in \Clo_{M + N} (\ab{A})$ and $b_1, \ldots, b_M \in
              A$ such that
              \[
              q (a_1, \ldots, a_N) =  t (a_1, \ldots, a_N, b_1, \ldots, b_M)
              \]
              for all $a_1, \ldots, a_N \in A$.
              Since $t \in \Clo_{N+M} (\ab{A})$, there is
              a polynomial $p \in \Clop (F \cup \{ x_1 + x_2, - x_1, 0 \})
               \cap \Kxn{N+M}$ 
               such that $t = p^{\ab{K}, N+M}$. Then by~\eqref{eq:U}, $p \in L \, \Clop (H)$,
               and therefore, there is $l \in \N$ such that
               $p = \sum_{i=1}^l p_i$ with $p_i \in \Clop (H)$.
               We let
               \[
                   \begin{array}{rcl}
                       I  & := & \{ i \in \{1,\ldots, l \} : p_i \text{ contains all the variables }x_1, \ldots, x_N \}, \\
                       J  & := & \{ 1, \ldots, l \} \setminus I.
                   \end{array}
               \]
               For $i \in I$, $p_i$ contains at least $n^{k-1} + 1$ variables, and therefore
               $p_i$ induces the $0$-function on $\ab{A}$.
               Thus
               \[
               p^{\ab{K}, N+M} = \sum_{i \in J} p_i^{\ab{K}, N + M}.
               \]
               For $i \in J$, let
               \[
               r_i (x_1, \ldots, x_N)  := p_i (x_1, \ldots, x_N, b_1, \ldots, b_M) \in \Kxn{N}.
               \]
               Then we have
               \[
               q (a_1, \ldots, a_N) = \sum_{i \in J} r_i^{\ab{K},N} (a_1, \ldots, a_N)
               \]
               for all $a_1, \ldots, a_N \in A$. Since $q$ is absorbing, $\sum_{i \in J} r_i (x_1, \ldots, x_N)$ induces an
               absorbing function on $K$. By Lemma~\ref{lem:abs1}, $\sum_{i \in J} r_i (x_1, \ldots, x_N)$
               induces the same function as $H_{\{1, \ldots, N\}} (\sum_{i \in J} r_i (x_1, \ldots, x_N))$.
               For each $i \in J$, $p_i$ does not contain all the variables $x_1, \ldots, x_N$. Thus $r_i$ has no monomial
               that contains all the variables $x_1, \ldots, x_N$, and therefore the sum $\sum_{i \in J} r_i (x_1, \ldots, x_N)$
               does not contain such a monomial, either. Hence $H_{\{1, \ldots, N\}} (\sum_{i \in J} r_i (x_1, \ldots, x_N)) = 0$.
               Therefore $\sum_{i \in J} r_i  (x_1, \ldots, x_N)$ induces the $0$-function on $K$,
               which implies $q = 0$.
        \end{proof}
                   
      \section{Proofs of the main results} \label{sec:proofs}
      \emph{Proof of Theorem~\ref{thm:fund2}:}
         As a nilpotent algebra in a congruence modular variety,
      $\ab{A}$ has a Mal'cev term (see Theorem~6.2 of \cite{FM:CTFC} and
         the remarks after the proof of Corollary~7.2, cf. \cite[Theorem~2.7]{Ke:CMVW}).
      We choose an element $o \in A$ and let
      $L = \langle 0_A = \alpha_0, \alpha_1 , \cdots ,\alpha_h = 1_A \rangle$
      be a maximal chain in the congruence lattice of $\ab{A}$.
      By Lemma~\ref{lem:maxchain}, the abelian group associated with $\ab{A}$, $L$ and $o$ is elementary abelian,
      and therefore we can
      use Theorem~\ref{thm:expand} to expand $\ab{A} = \algop{A}{(f_i)_{i \in I}}$ with operations $+$ and $-$
      and thereby obtain an $h$-nilpotent expanded group
      $\ab{V} := \algop{A}{+,-,0, (f_i)_{i \in I}}$ with elementary abelian group reduct.
      Then by Theorem~\ref{thm:boundabs}, all nonzero absorbing polynomial functions of $\ab{V}$ are of arity
      at most $s = \big(m (q - 1)\big)^{h-1}$.
      Hence by Lemma~\ref{lem:snpeg}, $\ab{V}$ is $s$-supernilpotent, and then
      by Lemma~\ref{lem:reduct}, its reduct $\ab{A}$
      is also $s$-supernilpotent. The claim on the free spectrum now follows from Lemma~\ref{lem:snp}. \qed

      \emph{Proof of Corollary~\ref{cor:cmv}:}
      As a nilpotent algebra in a congruence modular variety,
      $\ab{A}$ has a Mal'cev term.
      We write $\ab{A} = \prod_{i=1}^n \ab{B}_i$ with
      each $\ab{B}_i$ of prime power order. 
      By Theorem~\ref{thm:fund2}, each $\ab{B}_i$ is
      $s_i$-supernilpotent with $s_i = (m (|B_i| - 1))^{h_i - 1}$,
      where $h_i$ is the height of the congruence lattice of $\ab{B}_i$.
      As a nilpotent algebra in a congruence modular variety, $\ab{B}_i$
      is congruence uniform \cite[Corollary~7.5]{FM:CTFC}, which implies $h_i \le \log_2 (|B_i|)$.
      Since $|B_i| \le |A|$, we have $s_i \le s$, and therefore
      each factor $\ab{B}_i$ is $s$-supernilpotent.
      Hence $\ab{A}$ is $s$-supernilpotent. The claim on the free spectrum again
      follows from Lemma~\ref{lem:snp}. \qed 

      \emph{Proof of Corollary~\ref{cor:dec}:}
      For proving~\eqref{it:dec1}, we assume that $\ab{A}$ has small
      free spectrum. Then from Lemma~\ref{lem:snp}\eqref{it:s4}$\Rightarrow$\eqref{it:s2}, we obtain that $\ab{A}$ is nilpotent.
      By \cite[Theorem~3.14]{Ke:CMVW}, $\ab{A}$ is isomorphic to
      a direct product of algebras of prime power order. Now Corollary~\ref{cor:cmv} yields that $\ab{A}$ is $\big((m (|A|-1))^{(\log_2 (|A|)-1)}\big)$-supernilpotent and that the free spectrum $f_{\ab{A}}$ is of the form
      $f_{\ab{A}} (n) = 2^{p(n)}$ with $\deg(p) \le (m (|A|-1))^{(\log_2 (|A|)-1)}$.
      For proving~\eqref{it:dec2}, we assume that
      $\ab{A}$ is supernilpotent. Then from Lemma~\ref{lem:snp}\eqref{it:s1}$\Rightarrow$\eqref{it:s4}, we obtain that $\ab{A}$ has small free spectrum.
      Now we proceed as in \eqref{it:dec1}. \qed
      
      \section{Acknowledgements}
      The author thanks Sebastian Kreinecker for
      numerous discussions on clones of polynomials and
      Neboj\v{s}a Mudrinski and Jakub Opr\v{s}al for discussions
      on Section~\ref{sec:prelsnp}.
      \def\cprime{$'$}
\providecommand{\bysame}{\leavevmode\hbox to3em{\hrulefill}\thinspace}
\providecommand{\MR}{\relax\ifhmode\unskip\space\fi MR }
\providecommand{\MRhref}[2]{%
  \href{http://www.ams.org/mathscinet-getitem?mr=#1}{#2}
}
\providecommand{\href}[2]{#2}

    \end{document}